\documentclass{article}

\usepackage{titlesec}

\usepackage{amsmath}
\usepackage{amssymb}
\usepackage{empheq}
\usepackage{xcolor, color, colortbl}
\usepackage{framed}
\usepackage{pifont}
\usepackage{enumitem}
\usepackage{graphicx} 

\usepackage{nicefrac}
\usepackage{booktabs}
\usepackage{multirow}
\usepackage{rotating}
\usepackage{nccmath}

\usepackage[tikz]{bclogo}
\usepackage{wasysym}

\definecolor{mydarkblue}{rgb}{0,0.08,0.45}
\usepackage[
    linkcolor=mydarkblue,
    citecolor=mydarkblue,
    filecolor=mydarkblue,
    urlcolor=mydarkblue,
    pdfview=FitH]{hyperref}

\usepackage{url}
\usepackage{relsize}

\def\xx{{\boldsymbol x}}
\def\yy{{\boldsymbol y}}

\def\bb{{\boldsymbol b}}

\def\yy{{\boldsymbol y}}

\def\AA{{\boldsymbol A}}

\def\HH{{\boldsymbol H}}

\def\dif{\mathop{}\!\mathrm{d}}

\def\MP{\mu_{\mathrm{MP}}}

\def\RR{{\mathbb R}}
\def\EE{{\mathbb{E}\,}}

\DeclareMathOperator{\tr}{tr}
\def\defas{\stackrel{\text{def}}{=}}

\DeclareMathOperator*{\Span}{\mathbf{span}}

\newcommand*\mybluebox[1]{\colorbox{myblue}{\hspace{1em}#1\hspace{1em}}}

\DeclareMathOperator*{\argmin}{{arg\,min}}

\definecolor{myblue}{HTML}{D2E4FC}
\definecolor{Gray}{gray}{0.92}
    
\newtheorem{assumption}{Assumption}

\usepackage{thmtools}
\usepackage{thm-restate}

\usepackage{wrapfig}

\newcommand{\BlackBox}{\rule{1.5ex}{1.5ex}}  
\newenvironment{proof}{\par\noindent{\bf Proof\ }}{\hfill\BlackBox\\[2mm]}
\newtheorem{example}{Example} 
\newtheorem{theorem}{Theorem}
\newtheorem{lemma}[theorem]{Lemma} 
\newtheorem{proposition}[theorem]{Proposition} 
\newtheorem{remark}[theorem]{Remark}
\newtheorem{corollary}[theorem]{Corollary}
\newtheorem{definition}[theorem]{Definition}

\newif\ifhideproofs
\hideproofstrue 

\usepackage[nohyperref, accepted]{icml2020}

\icmltitlerunning{Average-Case Acceleration Through Spectral Density Estimation}

\usepackage{fancyhdr} 
\begin{document}
\twocolumn[
\icmltitle{Average-Case Acceleration Through Spectral Density Estimation}



\icmlsetsymbol{equal}{*}

\begin{icmlauthorlist}
\icmlauthor{Fabian Pedregosa}{equal,goo}
\icmlauthor{Damien Scieur}{equal,sait}
\end{icmlauthorlist}

\icmlaffiliation{sait}{Samsung SAIT AI Lab, Montreal}
\icmlaffiliation{goo}{Google Research}

\icmlcorrespondingauthor{Fabian Pedregosa}{\mbox{pedregosa@google.com}}

\icmlkeywords{optimization, acceleration, average-case}

\vskip 0.3in
]



\printAffiliationsAndNotice{\icmlEqualContribution} 

\begin{abstract}
We develop a framework for the average-case analysis of random quadratic problems and derive algorithms that are optimal under this analysis. This yields a new class of methods that achieve acceleration given a model of the Hessian's eigenvalue distribution. We develop explicit algorithms for the uniform, Marchenko-Pastur, and exponential distributions. These methods have a simple momentum-like update, in which each update only makes use on the current gradient and previous two iterates. Furthermore, the momentum and step-size parameters can be estimated without knowledge of the Hessian's smallest singular value, in contrast with classical accelerated methods like Nesterov acceleration and Polyak momentum. Through empirical benchmarks on quadratic and logistic regression problems, we identify regimes in which the the proposed methods improve over classical (worst-case) accelerated methods.
\end{abstract}

\section{Introduction} \label{sec:introduction}

The traditional analysis of optimization algorithms is a worst-case analysis \cite{nemirovski1995information,nesterov2004introductory}.
This type of analysis provides a complexity bound for \emph{any} input from a function class, no matter how unlikely.
However, hard-to-solve inputs might rarely occur in practice, in which case these complexity bounds might not be representative of the observed running time.

Average-case analysis provides instead the \emph{expected} complexity of an algorithm over a class of problems, and is more representative of its typical behavior.
While the average-case analysis is standard for analyzing sorting~\citep{knuth1997art} and cryptography~\citep{katz2014introduction} algorithms, little is known of the average-complexity of optimization algorithms. 



\begin{figure}
    \centering
    \includegraphics[width=0.9\linewidth]{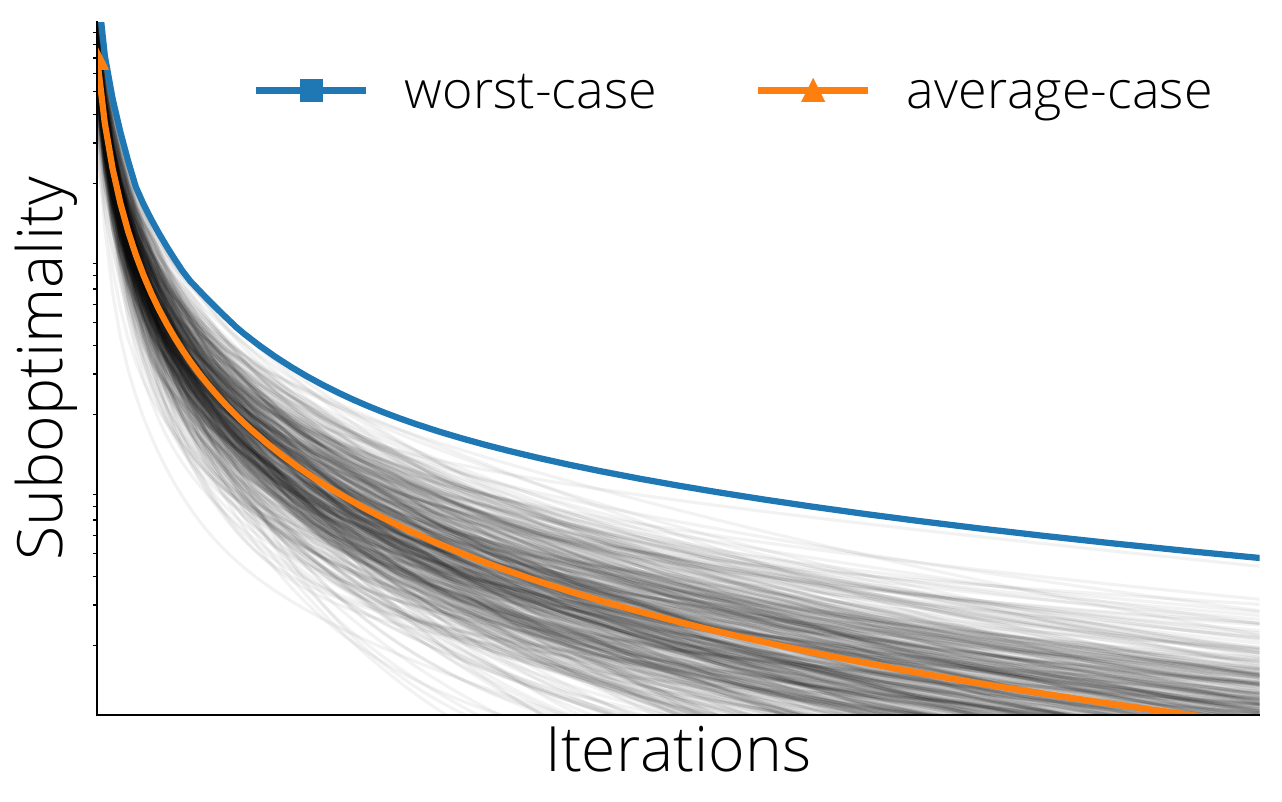}
    \caption{A worst-case analysis can lead to misleading results where the worst-case running times is much worse than the observed running time. Above: convergence of individual random square least squares in gray, while the average suboptimality (orange) is well below the worst-case (blue).}
    \label{fig:average}
\end{figure}

Our \textbf{first contribution} (\S\ref{sec:methods}) is to develop an average-case analysis of optimization methods on quadratic objectives. A crucial difference with the worst-case analysis is that it depends on the expected spectrum of the Hessian, rather on just the extremal eigenvalues.

While it is unfeasible to assume knowledge of the full spectrum, recent works have shown that the spectrum of large deep learning models is highly predictable and can be approximated using classical models from random matrix theory, such as the Marchenko-Pastur distribution~\citep{sagun2017empirical, martin2018implicit}.


A \textbf{second contribution} (\S\ref{sec:acceleration}) is to exploit this regularity to develop practical methods that are optimal under the average-case analysis. We consider different parametric models for the expected spectrum such as the Marchenko-Pastur, uniform, and exponential distributions, and derive for each model an average-case optimal algorithm.
These are all momentum-like methods where the hyper-parameter of the distribution can be estimated without knowledge of the smallest eigenvalue.

Finally, we compare these average-case optimal methods on synthetic and real datasets. We identify regimes where average-case exhibits a large computational gain with respect to respect to traditional accelerated methods.


\subsection{Related Work}

We comment on two of the main ideas behind the proposed methods.

\paragraph{Polynomial-Based Iterative Methods.} Our work draws heavily from the classical framework of polynomial-based iterative methods \citep{fischer1996polynomial} that can be traced back to the Chebyshev iterative method of \citet{flanders1950numerical}
and was later instrumental in the development of the celebrated conjugate gradient method~\citep{hestenes1952methods}.
More recently, this framework has been used to derive accelerated gossip algorithms \citep{berthier2018accelerated} and accelerated algorithms for smooth games \citep{azizian2020accelerating}, to name a few.
Although originally derived for the worst-case analysis of optimization algorithms, in this work we extend this framework to analyze also the average-case runtime.

In concurrent work, \citet{lacotte2020optimal} derive average-case optimal methods for a different class of methods where the update can be multiplied by a preconditioner matrix.

\paragraph{Spectral Density Estimation.}
The realization that deep learning networks behave as linear models in the infinite-width limit~\citep{jacot2018neural, novak2018bayesian}
has sparked a renewed interest in the study of spectral density of large matrices. 
The study of spectral properties of these very large models is made possible thanks to improved tools~\citep{ghorbani2019investigation} and more precise theoretical models \citep{jacot2019asymptotic, pennington2017nonlinear}.

\paragraph{Notation.}  Throughout the paper we denote vectors in lowercase boldface ($\xx$) and matrices in uppercase boldface letters ($\boldsymbol H$). Probability density functions and eigenvalues are written in Greek letters ($\mu, \lambda$), while polynomials are written in uppercase Latin letter ($P, Q$). We will sometimes omit integration variable, with the understanding that $\int \varphi \dif \mu$ is a shorthand for $\int \varphi(\lambda) \dif\mu(\lambda)$.


\section{Average-Case Analysis} \label{sec:methods}

In this section we introduce the average-case analysis framework for random quadratic problems.
The main result is Theorem~\ref{prop:rateconvergence}, which relates the expected error with other quantities that will be easier to manipulate, such as the residual polynomial. This is a convenient representation of an optimization method that will allow us in the next section to pose the problem of finding an optimal method as a best approximation problem in the space of polynomials.

Let $\HH \in \RR^{d \times d}$ be a random symmetric positive-definite matrix and $\xx^\star \in \RR^d$ a random vector. These elements determine the following (random) quadratic minimization problem
\begin{empheq}[box=\mybluebox]{equation*}\tag{OPT}\label{eq:quad_optim}
  \vphantom{\sum_0^i}\min_{\xx \in \RR^d} \Big\{ f(\xx) \defas\!\mfrac{1}{2}(\xx\!-\!\xx^\star)^\top\!\HH(\xx\!-\!\xx^\star) \Big\}\,.
\end{empheq}
Our goal is to quantify the expected error $\EE \|\xx_t - \xx^\star\|$, where $\xx_t$ is the $t$-th update of a first-order method starting from $\xx_0$ and $\EE$ is the expectation over the random $\HH, \xx_0, \xx^\star$.

\begin{remark} Problem \ref{eq:quad_optim} subsumes the quadratic minimization problem $\min_{\xx} \xx^\top\HH\xx + \bb^\top \xx + c$, but the notation above will be more convenient for our purposes.
\end{remark}

\begin{remark} The expectation in the expected error ${\EE \|\xx_t - \xx^\star\|^2}$ is over the inputs and not over any randomness of the algorithm, as would be common in the stochastic literature. In this paper we will only consider deterministic algorithms.
\end{remark}

To solve \ref{eq:quad_optim}, we will consider \emph{first-order methods}. These are methods in which the sequence of iterates $\xx_t$ is in the span of previous gradients, i.e.,
\begin{equation} \label{eq:first_order_methods}
    \xx_{t+1} \in \xx_0 + \Span\{ \nabla f(\xx_0), \ldots, \nabla f(\xx_t)  \}\, .
\end{equation}
This class of algorithms includes for instance gradient descent and momentum, but not quasi-Newton methods, since the preconditioner could allow the iterates to go outside of the span. Furthermore, we will only consider \emph{oblivious} methods, that is, methods in which the coefficients of the update are known in advance and don't depend on previous updates. This leaves out some methods that are specific to quadratic problems like conjugate gradient.

\paragraph{From First-Order Method to Polynomials.}
There is an intimate link between first order methods and polynomials that simplifies the analysis on quadratic objectives. Using this link, we will be able to assign to each optimization method a polynomial that determines its convergence. The next Proposition gives a precise statement:

\begin{restatable}{prop}{linkalgopolynomial}
    \label{prop:link_algo_polynomial} \citep{hestenes1952methods}
    Let $\xx_t$ be generated by a first-order method. Then there exists a polynomial $P_t$ of degree $t$ such that $P_t(0) = 1$ that verifies
    \begin{equation}\label{eq:polynomial_iterates}
        \vphantom{\sum^n}\xx_{t}-\xx^\star = P_t(\HH)(\xx_0-\xx^\star)~.
    \end{equation}
\end{restatable}

Following \citet{fischer1996polynomial}, we will refer to this polynomial $P_t$ as the \emph{residual polynomial}.

\begin{example}
(\textbf{Gradient descent}). The residual polynomial associated with the gradient descent method has a remarkably simple form. Subtracting $\xx^\star$ on both sides of the gradient descent update $\xx_{t+1} = {\xx_t - \gamma\HH(\xx-\xx^\star)}$ gives
\begin{align}
    \xx_{t+1}-\xx^\star &= (I-\gamma\HH)(\xx_t-\xx^\star)\\
    &= \cdots = (\boldsymbol{I}-\gamma\HH)^t(\xx_0-\xx^\star)
\end{align}
and so the residual polynomial is $P_t(\lambda) = (1-\gamma\lambda)^t$.
\end{example}

A convenient way to collect statistics on the spectrum of a matrix is through its \emph{empirical spectral distribution}, which we now define.

\begin{definition}
(\textbf{Empirical/Expected spectral distribution}). 
Let $\HH$ be a random matrix with eigenvalues $\{\lambda_1, \ldots, \lambda_d\}$. The \textbf{empirical spectral distribution} of $\HH$, called ${\mu}_{\HH}$, is the probability measure
\begin{equation}\label{eq:empirical_spectral_density}
    \mu_{\HH}(\lambda) \defas {\textstyle{\frac{1}{d}\sum_{i=1}^d}} \delta_{\lambda_i}(\lambda) ~,
\end{equation}
where $\delta_{\lambda_i}$ is the Dirac delta, a distribution equal to zero everywhere except at $\lambda_i$ and whose integral over the entire real line is equal to one.

Since $\HH$ is random, the empirical spectral distribution $\mu_{\HH}$ is a random measure. Its expectation over $\HH$,
\begin{equation}
\mu \defas \EE_{\HH}[\mu_{\HH}]\,,
\end{equation}
is called the \textbf{expected spectral distribution}.
\end{definition}



\begin{example}[Marchenko-Pastur distribution]\label{example:mp}
Consider a matrix $\AA \in \RR^{n \times d}$, where each entry is an i.i.d. random variables with mean zero and variance $\sigma^2$. Then it is known that the expected spectral distribution of $\HH = \frac{1}{n}\AA^\top\!\AA$ converges to the Marchenko-Pastur distribution \citep{marchenko1967distribution} as $n$ and $d \to \infty$ at a rate in which $d / n \to r \in (0, \infty)$.
The Marchenko-Pastur distribution $\MP$ is defined as
\begin{equation} \label{eq:mp_distribution}
     \max\{1 - \tfrac{1}{r}, 0\}\delta_0(\lambda)+\frac{\sqrt{(L - \lambda)(\lambda - \ell)}}{2 \pi \sigma^2 r \lambda}1_{\lambda \in [\ell, L]}\dif\lambda,
\end{equation}
where $\ell\defas \sigma^2(1 - \sqrt{r})^2$, $L \defas \sigma^2(1 + \sqrt{r})^2$ are the extreme nonzero eigenvalues and $\delta_0$ is a Dirac delta at zero (which disappears for $r \geq 1$).

In practice, the Marchenko-Pastur distribution is often a good approximation to the spectral distribution of high-dimensional models, even for data that might not verify the i.i.d. assumption, like the one in Figure~\ref{fig:mp}.
\end{example}

\begin{figure}
\centering\includegraphics[width=0.8\linewidth]{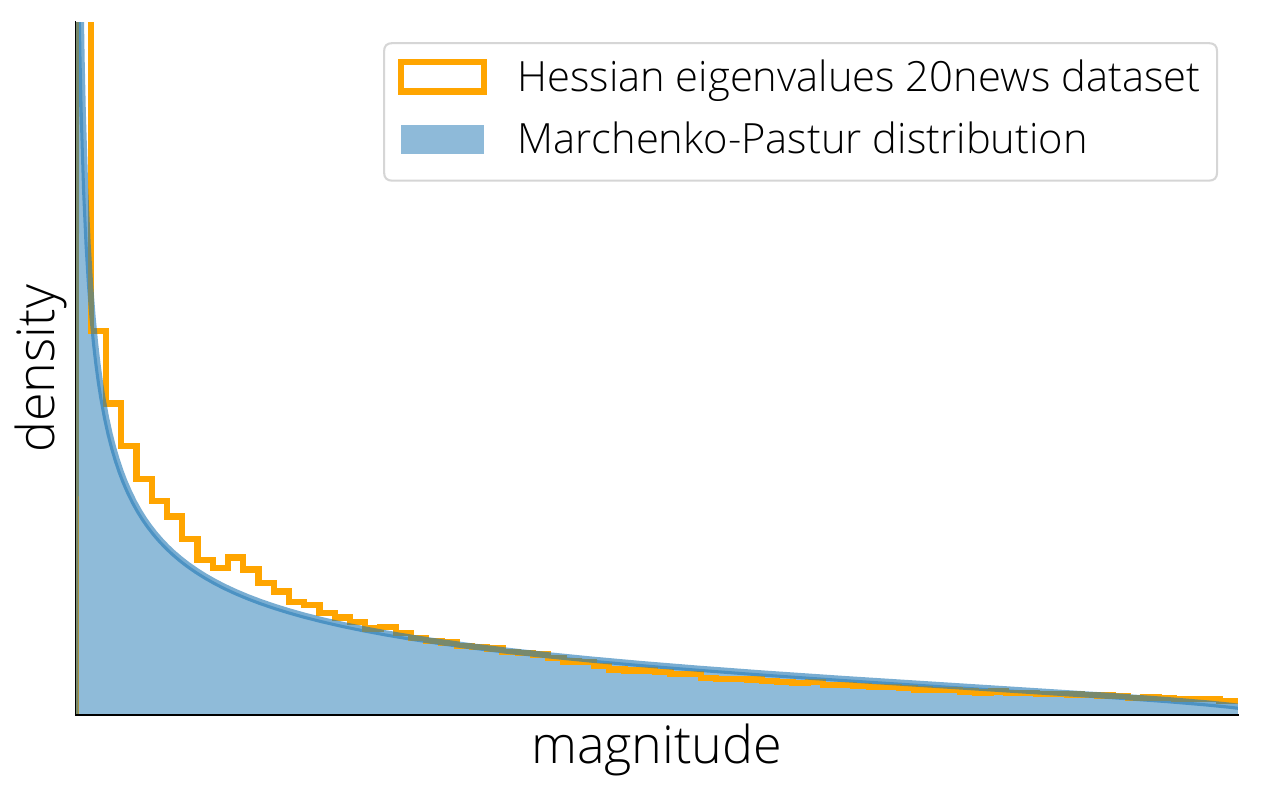}\label{fig:mp}
\caption{Eigenvalue distribution of high dimensional problems can often be approximated with simple models. Above, a fit to the Marchenko-Pastur distribution (blue) of the Hessian's eigenvalues histogram (yellow) in a least square problem with data the News20 dataset (20k features,~\citet{keerthi2005modified}). }
\end{figure}

Before presenting the main result of this section, we state one simplifying assumption on the initialization that we make throughout the rest of the paper.
\begin{assumption}\label{ass:independence}
We assume that $\xx_0\!-\!\xx^\star$ is independent of $\HH$ and
\begin{equation} \label{eq:assumption_initialization}
    \EE (\xx_0 -\xx^\star)(\xx_0 - \xx^\star)^\top =  R^2 \boldsymbol{I}\,.
\end{equation}
\end{assumption}
This assumption is verified for instance when both $\xx_0$ and $\xx^{\star}$ are drawn independently from a distribution with scaled identity covariance. It is also verified in a least squares problem of the form $\min_{\xx} \|\AA \xx - \bb\|^2$, where the target vector verifies $\bb = \AA \xx^\star$.\footnote{After the first version of this paper appeared online, \citet{paquette2020} generalized these results to least squares problems with a noisy target vector $\bb = \AA \xx^\star + \boldsymbol{\eta}$.}

\begin{framed}
\begin{restatable}{thm}{rateconvergence} \label{prop:rateconvergence}
Let $\xx_t$ be generated by a first-order method, associated to the polynomial $P_t$. Then we can decompose the expected error at iteration $t$ as
\begin{empheq}{equation}\label{eq:error_norm_x}
  \vphantom{\sum_0^i}\mathbb{E}\|\xx_t-\xx^\star\|^2 = {\color{brown}\overbrace{R^2\vphantom{R_t}}^{\text{initialization}}}\int {\color{teal}\underbrace{P_t^2}_{\text{algorithm}}} {\color{purple}\overbrace{\dif\mu}^{\text{problem}}}
  \,.
\end{empheq}
\end{restatable}
\end{framed}

This identity represents the expected error of an algorithm in terms of three interpretable quantities:
\begin{enumerate}[leftmargin=*]
    \item The {\color{brown}distance to optimum at initialization} enters through the constant $R$, which is the diagonal scaling in Assumption~\ref{ass:independence}.
    \item The {\color{teal} optimization method} enters in the formula through its residual polynomial $P_t$. The main purpose of the rest of the paper will be to find optimal choices for this polynomial.
    \item The {\color{purple} difficulty of the problem class} enters through the expected spectral distribution $\mu \defas \EE_{\HH}[\mu_{\HH}]$.
\end{enumerate}

\ifhideproofs
\else
\input{sections/proofs/average_case.tex}
\fi

\begin{remark}
Although we will not use it in this paper, similar formulas can be derived for the objective and gradient suboptimality:
\begin{align}
    &\mathbb{E}[f(\xx_t)-f(\xx^\star)] = R^2 \int_{\mathbb{R}} P^2_t(\lambda) \lambda \dif\mu(\lambda)\label{eq:error_f}\\
    &\mathbb{E}\,\|\nabla f(\xx_t)\|^2 = R^2 \int_{\mathbb{R}} P^2_t(\lambda) \lambda^2 \dif\mu(\lambda)\,. \label{eq:error_norm_grad}
\end{align}
\end{remark}

\section{Average-Case Acceleration} \label{sec:acceleration}

The framework developed in the previous section opens the door to exploring the question of optimality with respect to the average-case complexity. Does a method exist, that is optimal in this analysis? If so, what is this method?

In this section we will give a constructive answer to this last question. We will first introduce some concepts from the theory of orthogonal polynomials which will be necessary to develop optimal methods.

\begin{definition} \label{def:orthogonal_polynomial}
    Let $\alpha$ be a non-decreasing function such that $\int Q \dif\alpha$ is finite for all polynomials $Q$.
    We will say that the sequence of polynomials $P_0, P_1, \ldots$ is orthogonal with respect to $\dif\alpha$ if $P_i$ has degree $i$ and 
\begin{equation}\label{eq:optimal_orthogonal_polynomials}
    \int_{\mathbb{R}} P_i\, P_j\dif\alpha \begin{cases}
    = 0 & \text{if } i\neq j \\
    > 0 & \text{if } i = j
    \end{cases}.
\end{equation}
Furthermore, if they verify $P_i(0) = 1$ for all $i$, we call these {\bfseries residual orthogonal} polynomials.
\end{definition}

While a degree $t$ polynomial needs in principle $t$ scalars to be fully specified, residual orthogonal polynomials admit a compressed recursive representation in which each polynomial can be represented using only \emph{two} scalars and the previous two residual orthogonal polynomials. This compressed representation is known as the \emph{three-term recurrence} and will be crucial to ensure optimal methods enjoy the low memory and low per-iteration cost of momentum methods. The following lemma states this property more precisely.

\begin{lemma}[Three-term recurrence]\label{thm:recurence_orthogonal_polynomials} \citep[\S 2.4]{fischer1996polynomial}
Any sequence of residual orthogonal polynomials $P_1, P_2, \ldots$ verifies the following three-term recurrence
\begin{equation}\label{eq:recurence_orthogonal_polynomials}
    P_{t}(\lambda)
    = (a_t + b_t \lambda) P_{t-1}(\lambda) + (1\!-\!a_{t})P_{t-2}(\lambda)
\end{equation}
for some scalars $a_t, b_t$, with $a_{0} = a_1 = 1$ and $b_{0} = 0$.
\end{lemma}

\begin{remark}
Although there exist algorithms to compute the coefficients $a_t$ and $b_t$ recursively (e.g., \citet[Algorithm 2.4.2]{fischer1996polynomial}), numerical stability and computational costs make these methods unfeasible for our purposes. In Sections \ref{sec:exponential}--\ref{sec:uniform} we will see how to compute these coefficients for specific distributions of $\mu$.
\end{remark}

We have now all ingredients to state the main result of this section, a simple algorithm with optimal average-case complexity.

\begin{framed}
\begin{restatable}{thm}{optimalpolynomial}\label{thm:optimal_polynomial}
    Let $P_t$ be the residual orthogonal polynomials of degree $t$ with respect to the weight function $\lambda \dif\mu(\lambda)$, and let $a_t, b_t$ be the constants associated with its three-term recurrence. Then the algorithm
    \begin{equation}
    \begin{split}\label{eq:optimal_method}
        \xx_{t} = \xx_{t-1} &+ (1-a_t)(\xx_{t-2} - \xx_{t-1})\\
        &\qquad+  b_t \nabla f(\xx_{t-1})\,,
        \end{split}
    \end{equation}
    has the smallest expected error $\EE \|\xx_t - \xx^\star\|$ over the class of oblivious first-order methods.
    Moreover, its expected error is
    \begin{equation}\label{eq:optimal_rate}
        \mathbb{E} \|\xx_t-\xx^\star\|^2 =  R^2 \int_{\mathbb{R}} P_t  \dif\mu\,.
    \end{equation}
\end{restatable}
\end{framed}
\ifhideproofs
\else
\input{sections/proofs/optimal_polynomial.tex}
\fi

\begin{remark}
Algorithm \eqref{eq:optimal_method}, although being optimal over the space of all first-order methods, does not require storage of previous gradients. Instead, it has a very convenient momentum-like form and only requires storing two $d$-dimensional vectors.
\end{remark}

\begin{remark}
\textbf{(Relationship with Conjugate Gradient).} The derivation of the proposed method bears a strong resemblance with the Conjugate Gradient method~\citep{hestenes1952methods}. One key conceptual difference is that the conjugate gradient constructs the optimal polynomial for the \emph{empirical spectral distribution}, while we construct a polynomial that is optimal only for the expected spectral distribution. A more practical advantage is that the proposed method is more amenable to non-quadratic minimization.
\end{remark}

The previous theorem gives a recipe to construct an optimal algorithm from a sequence of residual orthogonal polynomials.
However, in practice we may not have access to the expected spectral distribution $\mu$, let alone its sequence of residual orthogonal polynomials.

The next sections are devoted to solving this problem through different parametric assumption on the spectral distribution: exponential (\S\ref{sec:exponential}), Marchenko-Pastur (\S\ref{sec:mp}), and uniform (\S\ref{sec:uniform}).

    \section{Optimal Method under the Exponential Distribution}  \label{sec:exponential}

In this section we assume that the expected spectral distribution follows an exponential distribution:
\begin{equation}
    \dif\mu = \frac{1}{\lambda_0} e^{-{\lambda}/{\lambda_0}}, \quad \lambda\in[0, \infty)\,,
\end{equation}
where $\lambda_0$ is a free parameter of the distribution. In this case the largest eigenvalue ($\equiv$ Lipschitz gradient constant) is unbounded, and so this setting can be identified the convex but non-smooth setting.

A consequence of Theorem~\ref{thm:optimal_polynomial} is that deriving the optimal algorithm is equivalent to finding the sequence of residual orthogonal polynomials with respect to the weight function $\lambda\dif\mu(\lambda)$.
Orthogonal (non-residual) polynomials for this weight function when $\dif\mu$ is an exponential distribution have been studied under the name of \textit{generalized Laguerre polynomials} \citep[p.\,781]{abramowitz1972handbook}. There is a constant factor between the orthogonal and \emph{residual} orthogonal polynomial, and so the latter can be constructed from the former by finding the appropriate normalization that makes the polynomial residual. This normalization is given by the following lemma:

\begin{restatable}{lemma}{laguerrepoly}\label{thm:laguerre_poly}
    The sequence of scaled Laguerre polynomials
    \begin{align}
        & P_0(\lambda) = 1\,, \quad P_1(\lambda) = 1-\mfrac{\lambda_0}{2}\lambda\,, \label{eq:recurence_alpha1_laguerre_normalized}\\
        & P_{t}(\lambda) = \left(\mfrac{2}{t+1} - \mfrac{\lambda_0}{t}\lambda\right) P_{t-1}(\lambda)  + \left(\mfrac{t-1}{t+1}\right) P_{t-2}(\lambda) \nonumber
    \end{align}
    are a family of residual orthogonal polynomials with respect to the measure $\frac{\lambda}{\lambda_0}e^{-{\lambda}/{\lambda_0}}$.
\end{restatable} 

From the above Lemma, we can derive the method with best expected error for the decaying exponential distribution. The resulting algorithm is surprisingly simple:

\begin{minipage}{0.99\linewidth}
\begin{bclogo}[logo=\hspace{17pt}]{Decaying Exponential Acceleration}
\vspace{0.5em}
\textbf{Input:} Initial guess $\xx_0$, $\lambda_0 > 0$

\textbf{Algorithm:} Iterate over $t=1...$
    \begin{equation}\label{algo:decaying_exponential}
        \xx_{t} = \xx_{t-1} +\frac{t-1}{t+1}(\xx_{t-1}-\xx_{t-2}) - \frac{\lambda_0}{t+1}\nabla f(\xx_{t-1}) \tag{EXP}
    \end{equation}
\end{bclogo}
\end{minipage}

\begin{remark}
In this distribution, and unlike in the MP that we will see in the next section, the largest eigenvalue is not bounded, so this method is more akin to subgradient descent and not gradient descent. Note that because of this, the step-size is decreasing.
\end{remark}

\begin{remark}
Note the striking similarity with the averaged gradient method of \citet{polyak1992acceleration}, which reads
\begin{align*}
    \boldsymbol{\theta}_{t} & = \yy_{t-1} - \gamma\nabla f(\boldsymbol{\theta})\\
    \xx_{t} & = \left(\mfrac{t-1}{t}\right)\xx_{t-1} + \mfrac{1}{t}\boldsymbol{\theta}
\end{align*}
This algorithm admits the following momentum-based form derived by \citet[\S2.2]{flammarion2015averaging}:
\begin{align*}
    {\color{purple} \yy_t} & {\color{purple} = t \xx_{t-1}\!-\!(t\!-\!1)\xx_{t-2} } \\
    \xx_t & = \xx_{t-1} +\frac{t-1}{t+1}(\xx_{t-1}-\xx_{t-2}) - \frac{\gamma}{t+1}\nabla f({\color{purple} \yy_t})~.    
\end{align*}
The difference between \ref{algo:decaying_exponential} and this formula is where the gradient is computed. In \ref{algo:decaying_exponential}, the gradient is evaluated at the current iterate $\xx_t$, while in Polyak averaging, the gradient is computed at the extrapolated iterate iterate $\yy_t$.
\end{remark}

\paragraph{Parameters Estimation.} The exponential distribution has a free parameter $\lambda_0$. Since the expected value of this distribution is ${1}/{\lambda_0}$, we can estimate the parameter $\lambda_0$ from the sample mean, which is $\frac{1}{d}{\tr}(\HH)$. Hence, we fit this parameter as $\lambda_0 = {d}/{\tr(\HH)}$.

\subsection{Rate of Convergence}

For this algorithm we are able to give a simple expression for the expected error.
The next theorem shows that it converges at rate $\mathcal{O}(1/t)$.
\begin{restatable}{lemma}{convergenceexponential}\label{convergence_exponential_gradient}
    If we apply Algorithm \eqref{algo:decaying_exponential} to problem \eqref{eq:quad_optim}, where the spectral distribution of $\HH$ is the decaying exponential $\dif\mu(\lambda) = e^{-{\lambda}/{\lambda_0}}$, then 
    \begin{equation}
        \mathbb{E} \|\xx_t-\xx^\star\|^{2} = \frac{R^2}{\lambda_0(t+1)}~.
    \end{equation}
\end{restatable}
\ifhideproofs
\else
\input{sections/proofs/convergence_exponential_gradient.tex}
\fi

In the case of the convex non-smooth functions, optimal algorithm achieves a (worst-case) rate which is  $O({1}/{\sqrt{t}})$ (see for instance \citep{nesterov2009primal}). We thus achieve acceleration, by assuming the function quadratic with the exponential as average spectral density.

\section{Optimal Method under the Marchenko-Pastur Distribution}  \label{sec:mp}

In this section, we will derive the optimal method under the Marchenko-Pastur distribution $\MP$, introduced in Example~\ref{example:mp}. As in the previous section, the first step is to construct a sequence of residual orthogonal polynomials.

\begin{restatable}{thm}{optimalmpcheby}\label{thm:optimal_mp_cheby}
The following sequence polynomials are orthogonal with respect to the weight function $\lambda \dif\MP(\lambda)$:
\begin{equation}
    \begin{split}
        &P_0(\lambda)=0\,,~P_1(\lambda) = 1~\delta_0=0, \rho = \mfrac{1+r}{\sqrt{r}}\\
        & \delta_t = \left(-\rho-\delta_{t-1}\right)^{-1}\\
        & P_t = \Big(- \rho\delta_t + \mfrac{\delta_t}{\sigma^2\sqrt{r}}\lambda\Big)  P_{t-1} + \left(1- \rho\delta_t\right)P_{t-2}\,. 
    \end{split}
\end{equation}
\end{restatable}
\ifhideproofs
\else
\input{sections/proofs/optimal_mp_cheby.tex}
\fi

\begin{figure*}
\centering
\begin{minipage}{0.9\linewidth}
\begin{bclogo}[logo=\hspace{17pt}]{Marchenko-Pastur Acceleration}
\vspace{0.5em}
\textbf{Input:} Initial guess $\xx_0$, MP parameters $r, \sigma^2$

\textbf{Preprocessing:} $\rho = \frac{1+r}{\sqrt{r}}$, $\delta_{0} = 0$, $\xx_1 = \xx_0-\frac{1}{(1+r)\sigma^2}\nabla f(\xx_0)$\\
~\\
\textbf{Algorithm (Optimal):} Iterate over $t=1...$
\begin{equation}\label{algo:mp_algo}\tag{MP-OPT}
    \delta_{t} = (-\rho - \delta_{t-1})^{-1}, \quad 
    \xx_{t} = \xx_{t-1} + \left(1 + \rho \delta_t\right)(\xx_{t-2} - \xx_{t-1}) + \delta_t \frac{ \nabla f(\xx_{t-1})}{\sigma^2\sqrt{r}}
\end{equation}

\textbf{Algorithm (Asymptotic variant):} Iterate over $t=1...$
    \begin{equation}\label{algo:mp_algo_asymptotic} \tag{MP-ASY}
        \xx_{t} = \xx_{t-1} - \min\{r^{-1}, r\}(\xx_{t-2} - \xx_{t-1}) - \frac{1}{\sigma^2} \min\{1, r^{-1}\} \nabla f(\xx_{t-1})
    \end{equation}
\end{bclogo}
\end{minipage}
\end{figure*}

We show in \ref{apx:mp} that these polynomials are shifted Chebyshev polynomials \emph{of the second kind}.
Surprisingly, the Chebyshev of the first kind are used to minimize over the worst case.
From the optimal polynomial, we can write the optimal algorithm for \ref{eq:quad_optim} when the spectral density function of $\HH$ is the Marchenko-Pastur distribution. By using Theorem~\ref{thm:optimal_polynomial}, we obtain the ``Marchenko-Pastur accelerated'' method, described in Algorithm \ref{algo:mp_algo}.

\begin{figure}[ht!]
    \centering
    \includegraphics[width=0.9\linewidth]{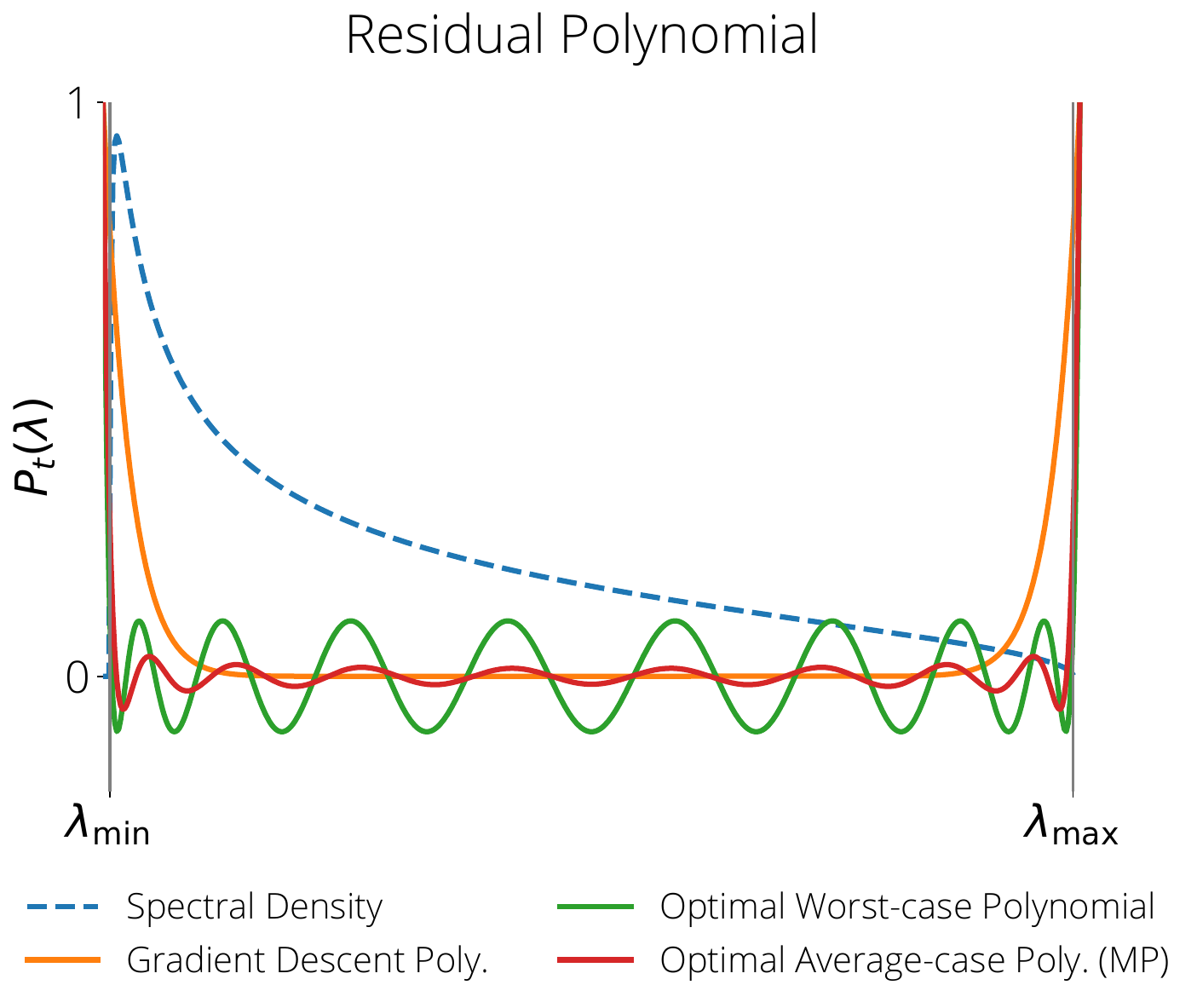}
    \label{fig:example_polynomials}
    \caption{\textbf{Residual polynomials} for gradient descent, the minimax optimal polynomial (Chebyshev of the first kind), and the optimal polynomial for the MP distribution, overlayed with a MP distribution in dashed lines. Gradient descent has the highest value at the edges, which is in line with its known suboptimal dependency on the condition number. The MP polynomial is closer to zero than the Chebyshev \emph{except} at the borders, which Chebyshev is optimal by construction.
    }\vspace{-1em}
\end{figure}



It's instructive to compare the residual polynomials of different methods to better understand their convergence properties. In Figure~\ref{fig:example_polynomials} we plot the residual polynomials for gradient descent, its classical worst-case analysis accelerated variant (Chebyshev polynomials) and the average-case analysis accelerated (under the MP distribution) variant above.  In numbers, the worst-case convergence rate (max value in interval)  is $\approx 0.88^{t}$ for Chebyshev and $\approx 0.93^t$ for MP, so Chebyshev's worst-case rate is 1.7 times smaller. However, the \textit{expected rate} is $\approx 0.67^t$ for MP and  $\approx 0.74^t$ for Chebyshev, making MP 1.4 times faster on average.

\paragraph{Asymptotic Behavior.}
The step-size and momentum terms of the Marchenko-Pastur accelerated method (\ref{algo:mp_algo}) depend on the time-varying parameter $\delta_t$. It's possible to compute the limit coefficients as $t \to \infty$. This requires to solve


\[
    \delta_\infty = (-{\rho - \delta_\infty})^{-1}, \;\; \text{thus } \; \delta_\infty = -\sqrt{r} \quad \text{or} \quad \delta_\infty = \frac{-1}{\sqrt{r}}.
\]
The sequence $\delta_t$ converges to $-\sqrt{r}$ when $r<1$, otherwise it converges to ${-1}/{\sqrt{r}}$. Replacing the value $\delta_t$ in Algorithm \ref{algo:mp_algo} by its asymptotic value gives the simplified variant \ref{algo:mp_algo_asymptotic}.

Algorithm \ref{algo:mp_algo} corresponds to a gradient descent with a variable momentum and step-size terms, all converging to a simpler one (Algorithm \ref{algo:mp_algo_asymptotic}). However, even if we assume that the spectral density of the Hessian is the Marchenko-Pastur distribution, we still need to estimate the hyper-parameters $\sigma$ and $r$. The next section proposes a way to estimate those parameters.

\paragraph{Hyper-Parameters Estimation.}
Algorithm \ref{algo:mp_algo} and Algorithm \ref{algo:mp_algo_asymptotic} require the knowledge of the parameters $r$ and $\sigma$. 
To ensure convergence, it is desirable to scale the parameters such that the largest eigenvalue $\lambda_{\max}$ lies inside the support of the distribution, bounded by $\sigma^2 (1 + \sqrt{r})^2$. This gives us one equation to set two parameters. We will get another equation by matching $r$ with the first moment of the MP distribution, which can be estimated as the trace of $\HH$. More formally, the two conditions reads
\begin{equation}
    \lambda_{\max}(\HH) = \sigma^2 (1 + \sqrt{r})^2\,,\quad \frac{\gamma}{d} \tr(\HH) = r\,.
\end{equation}
With the notation $\tau \defas \frac{1}{d}\tr(\HH)$, $ \defas \lambda_{\max}(\HH)$ the parameters $\sigma^2$ and $r$ are given by
\begin{equation}
\sigma^2 = \left( \frac{\sqrt{L} + \sqrt{\tau}}{L - \tau}\right)^2,\qquad r = \gamma \tau\,.
\end{equation}

\section{Optimal Method under the Uniform Distribution}  \label{sec:uniform}
We now focus on the (unnormalized) uniform distribution over $[\ell,L]$. This weight function is $1$ if $\lambda \in[\ell,L]$ and 0 otherwise.

We show in \ref{apx:uniform} that a sequence of orthogonal residual polynomials with respect to this density is a sequence of shifted \textit{Legendre polynomials}. Legendre polynomials are orthogonal w.r.t. the uniform distribution in $[-1,1]$ and are defined as
\begin{equation}\label{eq:legendre_polys}
    t \tilde Q_t(\lambda) = (2t-1) \lambda \tilde Q_{t-1}(\lambda) - (t-1) \tilde Q_{t-2}(\lambda)~.
\end{equation}
We detail in \ref{apx:uniform} the necessary translation and normalization steps to obtain the sequence of residual orthogonal polynomials. 


\begin{minipage}{0.99\linewidth}
\begin{bclogo}[logo=\hspace{17pt}]{Uniform Acceleration}
\vspace{0.5em}
\textbf{Input:} Initial guess $\xx_0=\xx_1$, $\ell$ and $L$.

\textbf{Init:} $e_{-1} = m_0= 0$.\\

\textbf{Algorithm:} Iterate over $t=1...$
    \begin{align}\label{algo:uniform}
    d_t & =  -\frac{L+\ell}{2} + e_{t-1} \nonumber\\
    e_t & = \frac{-(L-\ell)^2t^2}{4d_t(4t^2-1)} \tag{UNIF}\\
    m_t & = (d_t-e_t + m_{t-1} d_t e_{t-1})^{-1} \nonumber\\
    \xx_t & = \xx_{t-1} + \Big(1- m_t(d_t-e_t)\Big)(\xx_{t-2}-\xx_{t-1}) \nonumber\\
    &\qquad\quad~ + m_t \nabla f(\xx_{t-1}) \nonumber
    \end{align}
\end{bclogo}
\end{minipage}

Like the Marchenko-Pastur accelerated gradient, the parameters $\ell$ and $L$ can be estimated through the moment of the uniform distribution.

    \section{Experiments}  \label{sec:experiments}

We compare the proposed methods and classical accelerated methods on settings with varying degrees of mismatch with our assumptions.
We first compare them on quadratics generated from a synthetic dataset, where the empirical spectral density is (approximately) a Marchenko-Pastur distribution. We then compare these methods on another quadratic problem, generated using two non-synthetic datasets, where the MP assumption breaks down. Finally, we compare some the applicable methods in a logistic regression problem. We will see that the proposed methods perform reasonably well in this scenario, although being far from their original quadratic deterministic setting. A full description of datasets and methods can be found in \ref{apx:experiments}.

\paragraph{Methods.} We consider worst-case (solid lines) and average-case (dashed lines) accelerated methods. The method ``Modified Polyak'' runs Polyak momentum in the strongly convex case and defaults to the momentum method of \citep{ghadimi2015global} in the non-strongly convex regime.

\paragraph{Synthetic Quadratics.} We consider the least squares problem with objective function $f(\xx) = \|\AA \xx - \bb\|^2$, where each entry of $\AA \in \RR^{n \times d}$ and $\bb \in \RR^n$ are generated from an i.i.d. random Gaussian distribution. Using different ratios $d/n$ we generate problems that are convex ($\ell=0$) and strongly convex ($\ell > 0$). 

\paragraph{Non-Synthetic Quadratics.} The last two columns of Figure~\ref{fig:quadratic_experiments} compare the same methods, this time on two real (non-synthetic) datasets. We use two UCI datasets:
digits\footnote{{\tiny \url{https://archive.ics.uci.edu/ml/datasets/Optical+Recognition+of+Handwritten+Digits}}} ($n=1797, d = 64$), and
breast cancer\footnote{{\tiny\url{https://archive.ics.uci.edu/ml/datasets/Breast+Cancer+Wisconsin+\%28Diagnostic\%29}}} ($n=569, d=32$).

\paragraph{Logistic Regression.} Using synthetic i.i.d. data, we compare the proposed methods on a logistic regression problem. We generate two datasets, first one with $n > d$ and second one with $d < n$. Results are shown in Figure~\ref{fig:nn_experiments}.

\begin{figure*}
    \centering
    \includegraphics[width=\linewidth]{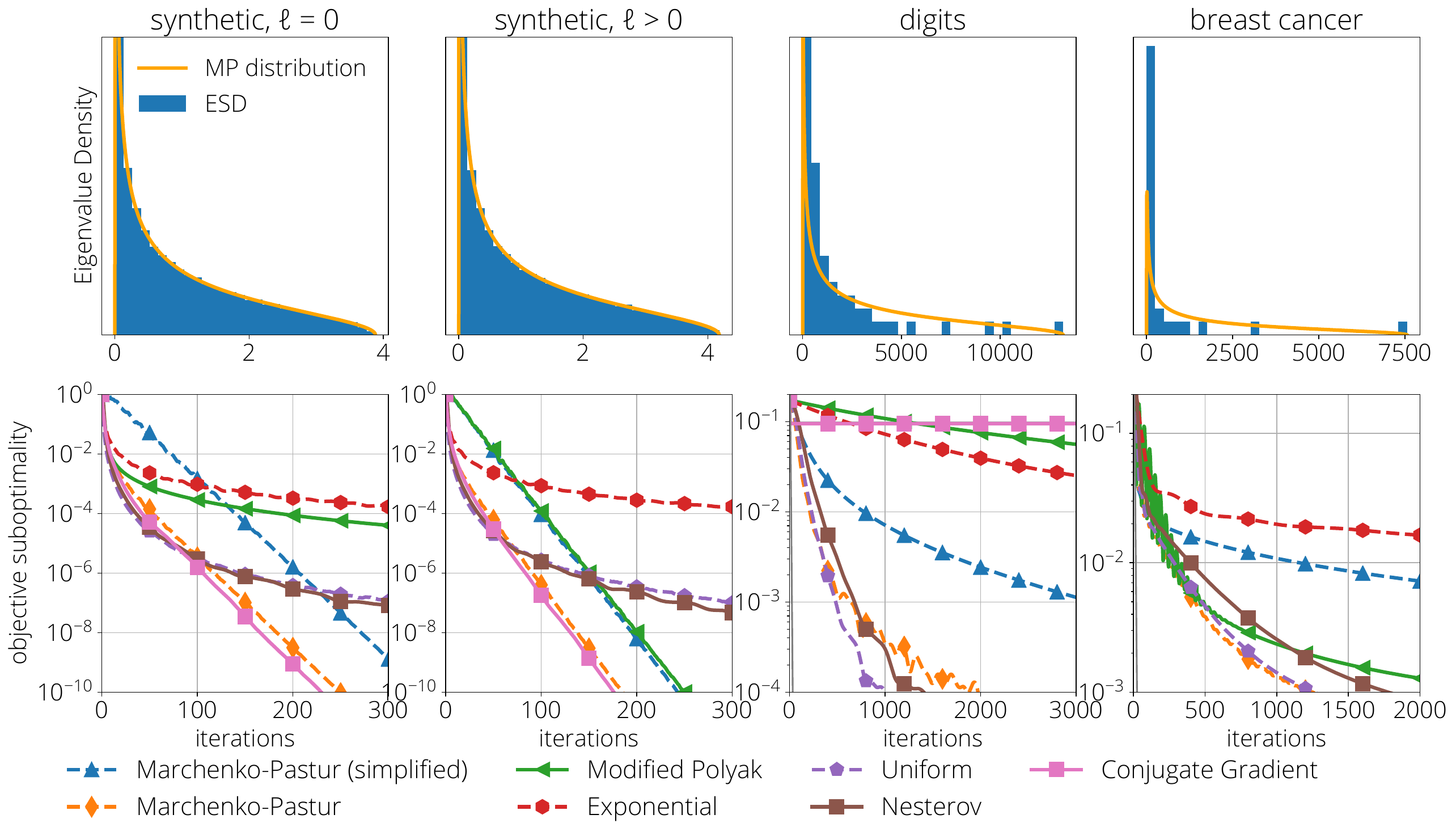}
    \caption{\textbf{Benchmark on quadratic objectives}. The top row shows the Hessian’s spectral
density overlaid with the MP distribution, while the bottom row shows the objective suboptimality as
a function of the number of iterations. In the first and second row, strongly convex and convex (respectively) synthetic problem with MP distribution. In all plots, the (Accelerated Gradient for) Marchenko-Pastur algorithm performs better than the classical worst-case analysis optimal method (Modified Polyak), and is close in performance with Conjugate Gradient one, which is optimal for the empirical (instead of expected) spectral density.
Note from the first plot that the Marchenko-Pastur method maintains a linear convergence rate even in the presence of zero eigenvalues. }
    \label{fig:quadratic_experiments}
\end{figure*}

\begin{figure*}
    \centering
    \includegraphics[width=\linewidth]{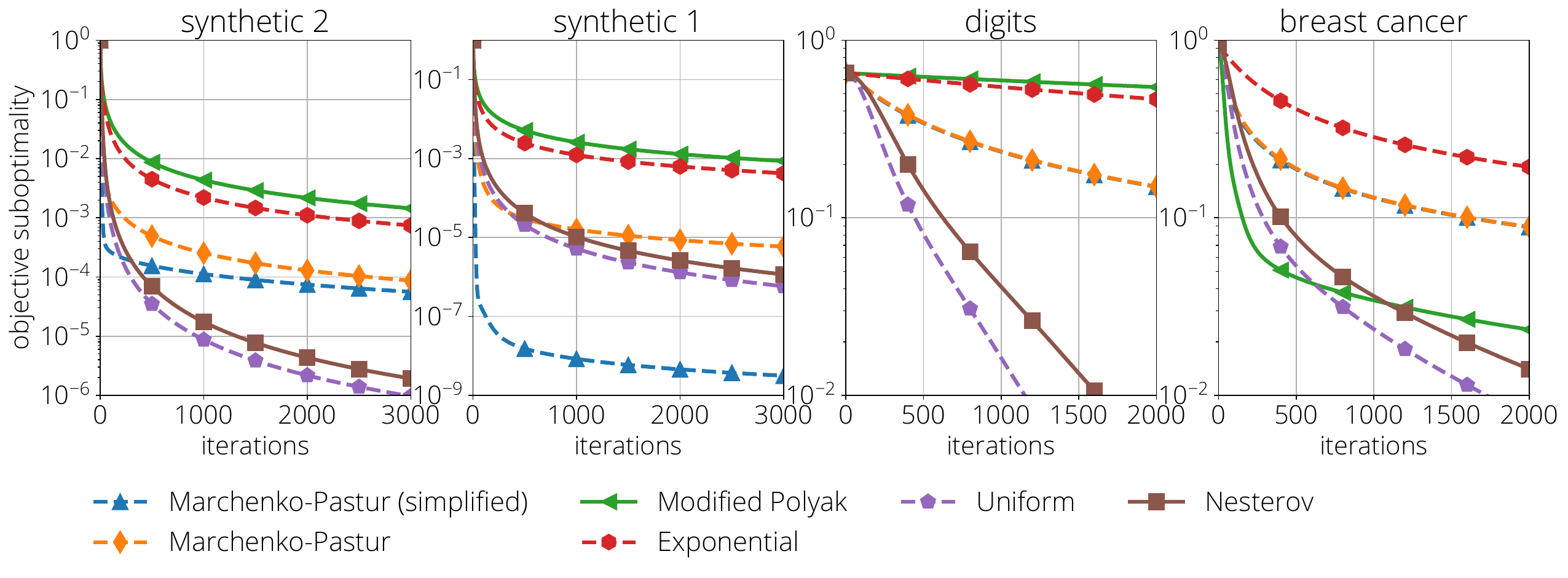}
    \caption{\textbf{Benchmark on logistic regression.} Objective suboptimality as
a function of the number of iterations on a logistic regression problem. 
Similarly to the results for quadratic objectives, we observe that average-case acceleration methods have an overall good performance, with the Uniform acceleration being best in 3 out of 4.
This provides some empirical evidence that average-case accelerated methods are applicable on problems beyond their original  quadratic setting.}
    \label{fig:nn_experiments}
\end{figure*}

\subsection{Discussion}

\paragraph{Importance of Spectral Distributions.} Average-case accelerated methods exhibit the largest gain when the eigenvalues follow the assumed expected spectral distribution. This is the case for the Marchenko-Pastur acceleration in the first two columns of Figure~\ref{fig:quadratic_experiments}, where its convergence rate is almost identical to that of conjugate gradient, a method that is optimal for the \emph{empirical} (instead of expected) spectral distribution.

\paragraph{Beyond Quadratic Optimization.} Results on 4 different logistic regression problems were mixed. On the one hand, as for quadratic objectives, the Uniform acceleration method has a good overall performance. However, contrary to the quadratic optimization results, Marchenko-Pastur acceleration has only mediocre performance in 3 out of 4 experiments. These results should also be contrasted with the fact that we did not adapt the methods to the non-quadratic setting. For example, a first step would be to restart the algorithm to accommodate a varying Hessian, or to perform backtracking line-search to choose the magnitude of the update. We leave these adaptations to be the subject of future work.

\section{Conclusion and Future Work}

In this work, we first developed an average-case analysis of optimization algorithms, and then used it to develop a family of novel algorithms that are optimal under this average-case analysis. We outline potential future work directions.

\emph{Modeling outlier eigenvalues}. Often the MP distribution fails to model accurately the outlier eigenvalues that arise in real data (e.g., last row of Figure~\ref{fig:quadratic_experiments}). A potential area for improvement is to consider distributions that can model these outlier eigenvalues.

\emph{Non-quadratic and stochastic extension}. As seen in Figure~\ref{fig:nn_experiments}, the methods are applicable and perform well beyond the quadratic setting in which they were conceived. We currently lack the theory to explain this observation.

\emph{Convergence rates and asymptotic behavior}. We have only been able to derive the average-case rate for the decaying exponential distribution. It would be interesting to derive average-case convergence rates for more methods. 

Regarding the asymptotic behavior, we have noted that some average-case optimal methods like Marchenko-Pastur acceleration converge asymptotically to Polyak momentum. 
After the first version of this work appeared online, \citep{scieur2020universal} showed that this is the case for almost all average-case optimal methods.

\section*{Acknowledgements}

We would like to thank our colleagues Gauthier Gidel, Adrien Taylor for thought-provoking early discussions, Crist\'obal Guzm\'an, Jeffrey Pennington,  Francis Bach, Pierre Gaillard and Rapha\"el Berthier for many pointers and dicussions, and Hossein Mobahi, Nicolas Le Roux, Courtney Paquette, Geoffrey Negiar, Nicolas Loizou, R\'emi Le Priol and Reza Babanezhad for fruitful discussion and feedback on the manuscript.

    \bibliography{biblio}
    \bibliographystyle{icml2020}
    
    \clearpage

    \appendix
    \onecolumn
\titleformat{\section}{\Large\bfseries}{\thesection}{1em}{}
\titleformat{\subsection}{\large\bfseries}{\thesubsection}{1em}{}
\gdef\thesection{Appendix \Alph{section}}
{\centering{\LARGE\bfseries Average-case Acceleration Through Spectral Density Estimation}

  \vspace{1em}
  \centering{{\LARGE\bfseries Supplementary Material}}

}
\vspace{2em}

    The appendix is organized as follows:
    \begin{enumerate}
        \item \ref{apx:proofs} proves results from \S\ref{sec:methods} and \ref{sec:acceleration}.
        \item \ref{sec:manip} develops tools manipulate orthogonal polynomials that will be used in later sections.
        \item \ref{apx:optimalpolys} develops the sequence of residual orthogonal polynomials for different weight functions.
        \item \ref{apx:experiments} contains some details on the experimental setting.
    \end{enumerate}

    \section{Proofs of Sections \ref{sec:methods} and \ref{sec:acceleration}}\label{apx:proofs}
We start by recalling the definition of expectation of a random measure:

\begin{definition}[\citet{tao2012topics}] \label{def:expected_spectral_density}
Given a random measure $\omega$, its expected measure $\EE\omega$ is the measure that satisfies
\begin{equation}
    \int \varphi \dif[\EE\omega] \,= \,\EE\!\int \varphi\dif\omega~,
\end{equation}
for any continuous $\varphi$ with compact support on $\mathbb{R}$. 
\end{definition}

\rateconvergence*
\begin{proof}
We prove this statement by combining the identity of Proposition~\ref{prop:link_algo_polynomial} and the definitions of empirical and expected spectral density (Definition~\ref{def:expected_spectral_density}). With this in mind, we can write the following sequence of identities that proves the statement:
\begin{align}
    \EE \|\xx_t - \xx^\star\|^2 &= \tr(P_t(\HH)^2 (\xx_0 - \xx^\star) (\xx_0 - \xx^\star)^\top ) &\text{ (Proposition~\ref{prop:link_algo_polynomial})}\\
    &= R^2 \,\EE \tr(P_t(\HH)^2) &\text{ (Assumption~\ref{ass:independence})} \\
    &= R^2 \EE \int P_t^2 \dif\mu_\HH &\text{ (Definition empirical spectral density)}\\
    &= R^2 \int P_t^2 \dif\mu &\text{ (Definition expected spectral density)}
\end{align}
\end{proof}

\optimalpolynomial*

\begin{proof}

This proof is structured into three parts. In the first part, we will show that $P_t$ is the optimal polynomial. In the second part, we will show that the iteration \eqref{eq:optimal_method} corresponds to this polynomial. In the third part we will show that this method achieves the expected convergence rate \eqref{eq:optimal_rate}.

{\bfseries \underline{Part 1}: Optimality of $P_t$.}
Let $P_t^\star$ be a residual polynomial of degree $t$ that minimizes the expected error
\begin{equation}
    P_t^\star \in \argmin_{P: \deg(P) \leq t, P(0)=1} \int P^2 \dif\mu\,.
\end{equation}
We will show that $P_t^\star$ and $P_t$ have the same error, that is, $\int (P_t^\star)^2 \dif\mu = \int P_t^2 \dif\mu$, which implies $P_t$ is optimal.

\begin{equation}
    \int (P_t^\star)^2 \dif\mu = \int (P_t + \lambda R_{t-1})^2 \dif\mu = \int P_t^2 \dif\mu + 2 \int P_t(\lambda) R_{t-1}(\lambda) \lambda \dif\mu(\lambda) + \int (\lambda R_{t-1}(\lambda))^2 \dif\mu(\lambda)
\end{equation}
In the last expression, the second term is zero by orthogonality of $P_t$, while the last one is non-negative since its the integral of a squared function. By optimality of $P^\star_t$, this last term cannot be strictly positive, as that would imply $P_t$ has a smaller error than $P^\star_t$. Hence, the last term must be zero and we have the desired $\int (P_t^\star)^2 \dif\mu = \int P_t^2 \dif\mu$.

{\bfseries \underline{Part 2}: Optimal method.}
We will now prove that $P_t$ is the residual polynomial associated with the method \eqref{eq:optimal_method}. We will show this by recursion.
    Removing $\xx^{\star}$ from both sides of Equation \eqref{eq:optimal_method} and using $\nabla f(\xx_{t-1}) = \HH(\xx_{t-1}-\xx^{\star})$ we have for $t=0$ and $t=1$:
    \begin{align*}
        \xx_0-\xx^{\star} & = \HH^0(\xx_0-\xx^{\star}) = P_0(\HH)(\xx_0-\xx^{\star})\\
        \xx_1-\xx^{\star} & = \xx_0-\xx^{\star} - b_1 \HH(\xx_0-\xx^{\star}) = P_1(\HH)(\xx_0-\xx^{\star}) \,.
    \end{align*}
    Assume now that the identity is true up to iteration $t-1$. Then for iteration $t$ we have
    \begin{align*}
        \xx_{t}-\xx^{\star} &= a_t(\xx_{t-1}-\xx^{\star}) - (1-a_t)(\xx_{t-2}-\xx^{\star}) - b_t \HH(\xx_{t-1}-\xx^{\star})\\
        &= \left(a_t P_{t-1}(\HH) - (1-a_t)P_{t-2}(\HH) - b_t \HH P_t(\HH)\right)(\xx_0-\xx^{\star}) = P_t(\HH)(\xx_0-\xx^\star)\,.
    \end{align*}

    {\bfseries \underline{Part 3}: Convergence rate.}
    We will now prove the simplified formula \eqref{eq:optimal_rate} for the convergence rate:
    \begin{equation}
        \int_{\mathbb{R}} P_t^2 \dif \mu = \int_{\mathbb{R}} P_t(\lambda) \left( \lambda \frac{P_t(\lambda)-P_t(0)}{\lambda} + P_t(0) \right) \dif\mu(\lambda).
    \end{equation}
    Let $Q_t(\lambda) \defas (1/\lambda)(P_t(\lambda)-P_t(0))$. This is a polynomial of degree $t-1$ since we took $P_t$, then we removed the independent term, then divided by $\lambda$. Since the sequence $\{P_i\}$ forms a basis of polynomial, we have that $P_t$ is orthogonal to all polynomials of degree less than $t$ w.r.t. $\lambda \mu(\lambda)$. In this case, $P_t$ is orthogonal to $Q_t$, thus
    \begin{equation}
        \int_{\mathbb{R}} P_t^2(\lambda) \dif\mu(\lambda) = \int_{\mathbb{R}} P_t(\lambda) Q_t(\lambda) \lambda \dif\mu(\lambda)  + \int_{\mathbb{R}} P_t(\lambda) P_t(0) \dif\mu(\lambda) = \int_{\mathbb{R}} P_t  \dif\mu,
    \end{equation}
    where the last equality follows from the fact that $P$ is a normalized polynomial, thus $P(0) = 1$.
\end{proof}

    \section{Manipulation over Polynomials} \label{sec:manip}
    
This section summarize techniques used to transform the recurrence of orthogonal polynomials.

In \ref{subsec:ortho_affine}, Theorem~\ref{thm:transform_orthogonal_polynomials}, we show how to transform the recurrence of a polynomial $Q_t$, orthogonal w.r.t. $\dif\mu(\lambda)$, into $P_t$, orthogonal to $\lambda\dif\mu(\lambda)$. It is a common situation where we have an explicit recurrence of orthogonal polynomials for the density $\mu$, but not for $\lambda\dif\mu(\lambda)$.

However, Theorem~\ref{thm:transform_orthogonal_polynomials} requires that the polynomials in the initial sequence $Q_t$ are monic, that is, a polynomial where the coefficient associated to the largest power is one. Unfortunately, most of explicit recurrences are \textit{not} monic. In \ref{subsec:transform_monic}, Proposition~\ref{prop:normalization_recurrence_monic}, we show a simple transformation of the recurrence of an orthogonal polynomial to make it monic.

Finally, to build an algorithm, we need to have residual polynomials, i.e., polynomials $P_t$ such that $P_t(0)=1$. In \ref{subsec:transformation_residual}, Proposition~\ref{prop:polynomial_conversion_normalized}, we give a technique that normalize the sequence to create residual polynomials.

An typical application is shown in Section~\ref{sec:uniform}. We start with the sequence of orthogonal polynomials orthogonal w.r.t. the uniform distribution, we apply Proposition~\ref{prop:normalization_recurrence_monic} to make the polynomials monic, then Theorem~\ref{thm:transform_orthogonal_polynomials}, and finally Proposition~\ref{prop:polynomial_conversion_normalized} to finally deduce the optimal algorithm.

\subsection{Computing the Orthogonal Polynomials w.r.t \texorpdfstring{$\lambda\dif\mu(\lambda)$}{lambda  dmu(lambda)} from \texorpdfstring{$\dif\mu(\lambda)$}{dmu(lambda)} using Kernel Polynomials} \label{subsec:ortho_affine}

The following theorem presents a procedure to transform a sequence of polynomials orthogonal w.r.t. the weight function $\dif\mu(\lambda)$ into a sequence that's orthogonal w.r.t. $\lambda \dif\mu(\lambda)$. 

\begin{theorem} \label{thm:transform_orthogonal_polynomials}
    \citep[Thm. 7]{gautschi1996orthogonal} Let $\{Q_i\}$ be a sequence of orthogonal polynomials w.r.t. the weight function $\lambda\dif\mu(\lambda)$ and $\xi$ a real value such that $Q(\xi)\neq 0$. Then the sequence of polynomials
    \begin{equation}
        P_i(\lambda) = \frac{1}{\lambda-\xi}\left( Q_{i+1}(\lambda) - \frac{Q_{i+1}(\xi)}{Q_{i}(\xi)} Q_{i}(\lambda)\right)
    \end{equation}
    is orthogonal w.r.t. the weight function $(\lambda-\xi)$.
    
    Moreover, if the recurrence for the $Q_t$'s reads
    \begin{equation}
        Q_{t}(\lambda) = (\alpha_t+\lambda)Q_{t-1}(\lambda) + \gamma_t Q_{t-2}(\lambda),
    \end{equation}
    (i.e., $Q_t$ is a sequence of \textit{monic} orthogonal polynomials), then the sequence $P_1, P_2, \ldots$ admits the following recurrence
    \begin{equation}
        \begin{cases}
            e_{-1} & = 0 \\
            d_t & = \alpha_t + e_{t-1} + \xi\\
            e_t & = {\widetilde c_{t+1}}/{d_t}\\
            P_t & = (x-\xi+d_t-e_t)Q_{t-1} + d_t e_{t-1} Q_{t-2}\,.
        \end{cases}
    \end{equation}
\end{theorem}
This theorem allows us to deduce the sequence of orthogonal polynomials $\{P_i\}$ from the knowledge of a sequence of orthogonal polynomials w.r.t. $\lambda\dif\mu(\lambda)$. Using this theorem recursively can be particularly useful. For example, if the sequence for the weight function $\mu$ is known, then Theorem~\ref{thm:transform_orthogonal_polynomials} allows us to deduce the optimal polynomial for the quantities \eqref{eq:error_norm_x}, \eqref{eq:error_f} and \eqref{eq:error_norm_grad}, since it requires orthogonal polynomials w.r.t. $\lambda^{\beta}\dif\mu(\lambda)$ for $\beta = 1,\,2,\,3$.

\subsection{Transformation into Monic Orthogonal Polynomials} \label{subsec:transform_monic}

Theorem~\ref{thm:transform_orthogonal_polynomials} requires the sequence of polynomials to be \textit{monic}, which means that the coefficient associated to the largest power is equal to one. In the recursion, this means the coefficient $b_t$ should be equal to one. The following proposition shows how to normalize the recurrence to end with monic polynomials.

\begin{proposition}\label{prop:normalization_recurrence_monic}
Let $\widetilde Q_t$ be defined as
\begin{equation}
    \widetilde Q_{t}(\lambda) = (\alpha_t + \beta_t \lambda) \widetilde Q_{t-1}(\lambda) + \gamma_t \widetilde Q_{t-2}(\lambda).
\end{equation}
Then, we can derive the sequence of monic orthogonal polynomials $Q_t$ through the recurrence
\begin{align}
    &Q_t(\lambda) =  (a_t +   \lambda) Q_{t-1}(\lambda) +  c_t  Q_{t-2}(\lambda), \\
    & \text{where}\qquad a_t =  \frac{\alpha_t}{\beta_t} \quad \text{and}\quad c_t = \frac{\gamma_t}{\beta_t \beta_{t-1}}.
\end{align}
\end{proposition}
\begin{proof}
    We start with the definition of $\widetilde Q_t$,
    \begin{equation}
        \widetilde Q_{t}(\lambda) = (\alpha_t + \lambda \beta_t) \widetilde Q_{t-1}(\lambda) + \gamma_t \widetilde Q_{t-2}(\lambda).
    \end{equation}
    Let $ \widetilde Q_t = \gamma_t Q_t$. Then,
    \begin{equation}
        Q_t = \frac{\gamma_{t-1}}{\gamma_t} \alpha_t \widetilde Q_{t-1}(\lambda) + \frac{\gamma_{t-1}}{\gamma_t}\beta_t \lambda\widetilde Q_{t-1}(\lambda) + \frac{\gamma_{t-2}}{\gamma_t}\gamma_t \widetilde Q_{t-2}(\lambda).
    \end{equation}
    In order to have $b_t = 1$, we need 
    \begin{equation}
    \frac{\gamma_{t-1}}{\gamma_t} = \frac{1}{\beta_t} \text{, and so } \frac{\gamma_{t-2}}{\gamma_t} = \frac{\gamma_{t-1}}{\gamma_t} \frac{\gamma_{t-2}}{\gamma_{t-1}} = \frac{1}{\beta_{t-1}\beta_{t}}\,.
    \end{equation}
\end{proof}

\subsection{Transformation into Residual Orthogonal Polynomials} \label{subsec:transformation_residual}

Building the optimal algorithm requires the polynomial to be \textit{residual}, that is, that it verifies $P_t(0)=1$. The next proposition shows how to transform a sequence of orthogonal polynomials $\widetilde{P}_t$ into a sequence of \textit{normalized} polynomials $P_t$
using an additional sequence $\delta_t$.

\begin{proposition}\label{prop:polynomial_conversion_normalized}
Let $\widetilde{P}_1, \widetilde{P}_2, \ldots $ be a sequence of polynomials that verifies the three-term recurrence
\begin{equation}
\widetilde P_{-1}(\lambda) = 0\,,\quad \widetilde{P}_0(\lambda) = \widetilde{P}_0\,,\quad
    \widetilde P_{t}(\lambda) = (\alpha_t + \beta_t \lambda)\widetilde P_{t-1}(\lambda) + \gamma_t \widetilde P_{t-2}(\lambda).
\end{equation}
    Then, if $\widetilde{P}_t(0) \neq 0$ for all $t$, the three-term recurrence
    \begin{equation}
        P_t(\lambda) =  (a_t + b_t \lambda)P_{t-1}(\lambda) + (1 - a_t) P_{t-2}(\lambda)\,,
    \end{equation}
    for the associated residual orthogonal polynomial $P_t(\lambda) = \frac{\widetilde{P}_t(\lambda)}{\widetilde{P}_t(0)}$ are given by
    \begin{equation}
        \delta_t = (\alpha_t+\gamma_t\delta_{t-1})^{-1} ~ (\text{with } \delta_0 = 0) \quad a_t = \delta_t\alpha_t, \quad b_t = \delta_t\beta_t\,.
    \end{equation}
\end{proposition}
\begin{proof}
    Let $P_t = (1/\gamma_t)\tilde P_t $. Then,
    \begin{equation}
        P_t(\lambda) = \tilde a_t \frac{\gamma_{t-1}}{\gamma_t} P_{t-1}(\lambda) + \tilde b_t \lambda\frac{\gamma_{t-1}}{\gamma_t} P_{t-1}(\lambda) + \tilde c_t\frac{\gamma_{t-1}}{\gamma_t}\frac{\gamma_{t-2}}{\gamma_{t-1}} P_{t-2}(\lambda)
    \end{equation}
    Let $\delta_t = \frac{\gamma_{t-1}}{\gamma_t}$. Since we need
    \begin{equation}
        c_t = 1-a_t \quad \Leftrightarrow \quad \tilde c_t\delta_t\delta_{t-1} = 1-a_t \delta_t.
    \end{equation}
    This gives the recurrence
    \begin{equation}
        \delta_t = \frac{1}{\tilde a_t + \tilde c_t \delta_{t-1}}.
    \end{equation} 
    It remains to compute the initial value, $\delta_0$. However,
    \begin{equation}
        P_0(\lambda) = 1 \quad \text{and} \quad P_1(\lambda) = \delta_1 a_1+b_1 \lambda.
    \end{equation}
    This means $\delta_0 = 0$ since we need $P_1(0) = \delta_1 a_1= 1$.
\end{proof}

    \section{Optimal Polynomials}\label{apx:optimalpolys}
    \subsection{Optimal Polynomials for the Exponential Distribution}
    \laguerrepoly*
\begin{proof}
We start with the definition of generalized Laguerre polynomials,
\begin{equation}
    \tilde P_0^{\alpha} = 1, \qquad \tilde P_1^{\alpha} = -\lambda+\alpha+1, \qquad \tilde P_{t}^{\alpha}(\lambda) = \frac{1}{t} \Big((2t-1+\alpha-\frac{\lambda}{\lambda_0})\tilde P_{t-1}^{\alpha}(\lambda) - (t-1+\alpha) \tilde P_{t-2}^{\alpha}(\lambda) \Big)
\end{equation}
which are orthogonal w.r.t. the weight function
\begin{equation}
    \mu(\lambda) = \lambda^{\alpha}e^{\frac{\lambda}{\lambda_0}}.
\end{equation}
In our case, we aim to find the sequence of orthogonal polynomial w.r.t. the weight function $\lambda e^{\lambda/\lambda_0}$, so we fix $\alpha = 1$. To make the notation lighter, we now remove the superscript. The polynomials of this sequence are not residual, thus we have to apply Proposition~\eqref{prop:polynomial_conversion_normalized}. From this proposition, we have that
\begin{equation}
    \delta_t \defas \frac{\tilde P_{t-1}(0)}{\tilde P_t(0)}.
\end{equation}
It is possible to show that $P_t(0) = t+1$, thus
\begin{equation}
    \delta_t = \frac{t}{t+1}.
\end{equation}
According to Proposition~\eqref{prop:polynomial_conversion_normalized}, using
\begin{equation}
    a_t = \delta_t\frac{2t-1+\alpha}{t},\qquad b_t=\frac{-\delta_t}{t}, \qquad c_t =  1-a_t
\end{equation}
in the sequence of residual orthogonal polynomial gives residual polynomials $P_t$ which are orthogonal w.r.t. the weight function $\lambda e^{\lambda/\lambda_0}$. By Theorem~\ref{thm:optimal_polynomial}, this leads to the optimal polynomial.
\end{proof}

\convergenceexponential*
\begin{proof}
    We assume $\lambda_0=1$ for simplicity. First, we use the useful summation property of Laguerre polynomials from \citep[equ. (22.12.6)]{abramowitz1972handbook},
    \[
        \tilde P^{\alpha+\beta+1}_t(x+y) = \sum_{i=0}^t\tilde P^{\alpha}_{i}(x) \tilde P^{\beta}_{t-i}(y).
    \]
    In particular,
    \[
        \tilde P^{1}_t(x) = \sum_{i=0}^t\tilde P^{0}_{i}(x).
    \]
    Thus, the following integral simplifies into
    \begin{equation}
        \int_{0}^{\infty} \tilde P^1_t(\lambda)^2 e^{-\lambda} \dif \lambda = \int_{0}^{\infty} \left(\sum_{i=0}^t\tilde P^{0}_{i}(\lambda)\right)\left(\sum_{i=0}^t\tilde P^{0}_{i}(\lambda)\right) e^{-\lambda} \dif \lambda = \sum_{i=0}^t \int_{0}^{\infty}  P^{0}_{i}(\lambda)^2 e^{-\lambda} \dif \lambda
    \end{equation}
    where the last equality follows from the orthogonality of Laguerre polynomials (see Theorem~\ref{thm:laguerre_poly}). Moreover, it can be shown that
    \begin{equation}
        \int_{0}^{\infty} P^{0}_{i}(\lambda)^2 e^{-\lambda} \dif \lambda = 1.
    \end{equation}
    This implies that
    \begin{equation}
        \int_{0}^{\infty} \tilde P^1_t(\lambda)^2 e^{-\lambda} \dif \lambda = t+1.
    \end{equation}
    However, $\tilde P^{1}_t$ is the non-normalized version of Laguerre polynomial, thus it needs to be divided by $\tilde P_t^1(0)$. It can be shown easily that $\tilde P_t^1(0) = t+1$, thus
    \[
        \int_{0}^{\infty} P_t^2(\lambda) e^{-\lambda} \dif \lambda = \frac{1}{(t+1)^2}\int_{0}^{\infty} (\tilde P_t)^2(\lambda) e^{-\lambda} \dif \lambda = \frac{1}{t+1}.
    \]
    The last step consists in evaluating the bound in Theorem~\ref{eq:error_norm_x} with $P(\lambda) = \frac{\tilde P^1(\lambda)}{\tilde P^1(0)}$, which gives
    \begin{equation}
        \mathbb{E} \|\xx_t-\xx^\star\|^{2} = R \int_{0}^{\infty} P_t^2(\lambda) e^{-\lambda} \dif \lambda = \frac{\|\xx_0-\xx^\star\|^2_2}{t+1}.
    \end{equation}
\end{proof}

\subsection{Optimal Polynomials for the Marchenko-Pastur Distribution}\label{apx:mp}

\optimalmpcheby*
    
\begin{proof}
We will derive these polynomials by identifying them with Chebyshev polynomials of the second kind. 

Let $U_t$ be the $t$-th Chebyshev polynomial of the second kind. These polynomials are defined through their three-term recurrence
\begin{equation}\label{eq:three_term_chebyshev_second_kind}
U_0(\xi) = 1, \qquad U_1(\xi) = 2\xi, \qquad U_t(\xi) = 2\xi U_{t-1}(\xi)-U_{t-2}(\xi)\,,
\end{equation}
and are known to be orthogonal to the semi-circle distribution $\sqrt{1 - \xi^2}$.

Let also $\ell \defas \sigma^2(1 - \sqrt{r})^2$, $L \defas \sigma^2(1 + \sqrt{r})^2$ be the edges of the Marchenko-Pastur distribution, introduced in Example~\ref{example:mp}
and consider the mapping from $[\ell, L]$ to $[-1, 1]$:
\begin{equation}
    \xi(\lambda) = - \frac{L+\ell}{L-\ell} + \frac{2}{L - \ell}\lambda
\end{equation}

We will show that the polynomials $U_i(\xi(\lambda))$ are orthogonal with respect to the weight function $\lambda \dif\MP(\lambda)$. To prove this, it will be sufficient to prove that the following integral is zero when $i\neq j$:
   \begin{equation}\label{eq:mp_integral_orthogonal}
    \int_{\ell}^{L} U_{i}\left(\xi(\lambda)\right)U_{j}\left(\xi(\lambda)\right)\sqrt{(L-\lambda)(\lambda-\ell)} \dif \lambda\,.
  \end{equation}

Because of the identity
\begin{equation}
    1 - \xi^2 = \frac{4 (L - \lambda)(\lambda - \ell)}{(L - \ell)^2}\,,
\end{equation}
the integral \eqref{eq:mp_integral_orthogonal} is proportional to 
  \begin{equation}
    \int_{-1}^{1} U_{i}\left(\xi\right)U_{j}\left(\xi\right)\sqrt{1 - \xi^2} \dif \xi\,,
  \end{equation}
  which is zero for $i\neq j$ by the orthogonality of $U_i$ with respect to the semi-circle distribution.
  Using Theorem~\ref{thm:optimal_polynomial}, we then have that $U_t(\xi(\lambda)) / U_t(\xi(0))$ are the polynomials that minimize the average-case convergence rate for the case in which the average spectral density is the Marchenko-Pastur distribution.
  
  We now seek the three-term recurrence associated with the residual polynomial $U_t(\xi(\lambda)) / U_t(\xi(0))$. A first step in this direction is to estimate the three-term recurrence of the numerator $\widetilde{P}(\lambda) \defas U_t(\xi(\lambda))$. 
  
  For this we will rewrite $\xi$ in the equivalent form
  \begin{equation}
      \xi(\lambda) = \frac{-\sigma^2(1 + r) + \lambda}{2 \sigma^2 \sqrt{r}}\,,
  \end{equation}
  which follows from the definition of $\ell$ and $L$.

  Using the definition of $\xi$ and the three-term recurrence \eqref{eq:three_term_chebyshev_second_kind} we have:
  \begin{align}
    & \widetilde{P}_0(\lambda) = U_0(\xi(\lambda)) = 1, \qquad \widetilde{P}_1(\lambda) = U_1(\xi(\lambda)) = 2 \xi(\lambda) = \frac{-\sigma^2(1 + r) + \lambda}{\sigma^2 \sqrt{r}} , \\
    & \qquad \widetilde{P}_t(\lambda) = U_t(\xi(\lambda)) = \left(\frac{-\sigma^2(1 + r) + \lambda}{\sigma^2 \sqrt{r}} \right)\widetilde{P}_{t-1}(\lambda) - \widetilde{P}_{t-2}(\lambda),
  \end{align}
    It only remains to normalize the $\widetilde{P}_t$ polynomials so they become residual polynomials. Using Proposition~\ref{prop:polynomial_conversion_normalized}, the normalization constants are given by 
    \begin{equation}
        \delta_0 = 0, \qquad  \delta_t = \left(-\frac{1+r}{\sqrt{r}} - \delta_{t-1}\right)^{-1}\,,\quad a_t =-\frac{\delta_t (1 + r)}{\sqrt{r}}\,,\quad b_t = \frac{\delta_t}{\sigma^2 \sqrt{r}}
    \end{equation}
    which coincides with the recurrence in the theorem statement.
\end{proof}

\subsection{Optimal Polynomials for the Uniform Distribution}\label{apx:uniform}

In this appendix, we develop a procedure that transform the recurrence of Legendre polynomials into the optimal algorithm for the uniform distribution. The transformation has the following structure:
\begin{eqnarray*}
    & & \tilde Q, ~ \text{orthogonal w.r.t. $\dif\mu(\lambda)$, (shifted Legendre, not monic)} \nonumber\\
    & & \quad \Big\downarrow \,\,  \text{(Proposition~\ref{prop:normalization_recurrence_monic})}\nonumber\\
    & & Q, ~ \text{orthogonal w.r.t. $\dif\mu(\lambda)$ and monic} \nonumber\\
    & & \quad \Big\downarrow \,\,  \text{(Theorem~\ref{thm:transform_orthogonal_polynomials})}\,\,\,\,~\nonumber\\
    & & \tilde P, ~  \text{orthogonal w.r.t. $\lambda \dif\mu(\lambda)$ but not normalized} \nonumber\\
    & & \quad \Big\downarrow \,\, \text{(Proposition~\ref{prop:polynomial_conversion_normalized})}\nonumber\\
    & & P, ~ \text{orthogonal w.r.t. $\lambda \dif\mu(\lambda)$ and normalized}
\end{eqnarray*}
All computation done, we end with the optimal Algorithm \ref{algo:uniform}.

These polynomials are orthogonal w.r.t. the uniform distribution over $[-1,1]$. In our case, we need over another, arbitrary interval $[\ell,L]$. The following corollary gives the \textit{shifted} version of Legendre polynomials using a simple linear mapping.

\begin{corollary}
    The polynomials defined through the following recursion are orthogonal with respect to the uniform distribution in the interval $[\ell,L]$. 
    \begin{align}
        \tilde q_t(\lambda) = - \left(\frac{2t-1}{t}\right) \frac{L+\ell}{L-\ell} \tilde q_{t-1}(\lambda) + \left(\frac{2t-1}{t}\right)\frac{2}{L-\ell}\lambda \tilde q_{t-1}(\lambda)  +  \left(\frac{1-t}{t}\right) \tilde q_{t-2}(\lambda)
    \end{align}
\end{corollary}    
\begin{proof}

We will prove this by showing that the above are Legendre polynomials, shifted by an affine function. Consider the linear function that maps $[-1,1]$ to $[\ell,L]$ ,
    \begin{align}
        y = \frac{L+\ell}{2} + \frac{L-\ell}{2}\lambda  \quad \Rightarrow \quad \lambda = \frac{2}{L-\ell}y - \frac{L+\ell}{L-\ell}.
    \end{align}
    Then, using \eqref{eq:legendre_polys} to compute the recurrence of $\tilde{Q(y)}$ and dividing the left- and right-hand side by $t$ yields the desired result.
\end{proof}

We now derive the monic version of the polynomial using Proposition~\ref{prop:normalization_recurrence_monic}, which reads
    \begin{align} \label{eq:monic_legendre}
        Q_t(\lambda) &= \left(-\frac{L+\ell}{2} + \lambda\right) Q_{t-1}(\lambda)-\frac{(L-\ell)^2(t-1)^2}{4(2t-1)(2t-3)}Q_{t-2}\,.
    \end{align}
    
Now, we use Theorem~\ref{thm:transform_orthogonal_polynomials} in the particular case where $\xi = 0$ so that we have the sequence of orthogonal polynomial w.r.t. the weight function $\lambda \dif \mu(\lambda)$.
\begin{equation}
    \begin{cases}
        e_{-1} & = 0 \\
        d_t & =  -\frac{L+\ell}{2} + e_{t-1} \\
        e_t & = \frac{-(L-\ell)^2t^2}{4d_t(4t^2-1)} \\
        P_t(\lambda) & = (\lambda+d_t-e_t)Q_{t-1}(\lambda) + d_t e_{t-1} Q_{t-2}(\lambda)
    \end{cases}
\end{equation}
The last step consists in the application of Proposition~\ref{prop:polynomial_conversion_normalized}.
\begin{equation}
        P_t(\lambda) = P_{t-1}(\lambda) + (1-a_t)(P_{t-2}(\lambda)-P_{t-1}(\lambda)) + \lambda b_t  P_{t-1}(\lambda),
\end{equation}
where
\begin{equation}
    a_t = m_t(d_t-e_t), ~ b_t = m_t ~ \text{and} ~ m_t =  \frac{1}{d_t-e_t + m_{t-1}(d_t e_{t-1})}.
\end{equation}

\section{Experiments}  \label{apx:experiments}
In the first column of Figure~\ref{fig:quadratic_experiments}, the ratio is $r = \frac{n}{d} = 0.9$ and there are zero eigenvalues, while in the second column the ratio is $1.1$ and there are none. 

The data matrix for the logistic regression experiment (Figure~\ref{fig:nn_experiments}, first two columns) is generated in the same way as for Figure~\ref{fig:quadratic_experiments}, except the target values is passed through a logistic function to have values in the interval $[0, 1]$.

We use the following procedure to estimate the moments of this distribution.
The $k$-th moment $\beta_k$ of the MP distribution can be computed from $\beta_k = \sigma^{2k}\sum_{i=0}^{k-1} \binom{k}{i} \binom{k-1}{i} \frac{r^{k-i}}{i+1}$ \citep[Lemma 3.1]{bai2010spectral}.

We can then match these moments with the empirical ones. For example, let $\hat{\beta_1} = \frac{1}{d}\tr(\HH)$, $\hat{\beta_2} = \frac{1}{d}\tr(\HH^2)$ be the first two empirical moments. 
Matching the first two moments results in 
\begin{equation}
    r = \frac{\beta_1^2}{\beta_2 - \beta_1^2}\;,\quad \sigma = \sqrt{\frac{\beta_2 - \beta_1^2}{\beta_1}}\,.
\end{equation}

\end{document}